\documentclass{amsart}
\usepackage{amsmath}
\usepackage{amsthm}
\usepackage{amsfonts}
\newtheorem{theorem}{Theorem}
\newtheorem{lemma}[theorem]{Lemma}
\theoremstyle{definition}
\newtheorem*{remark}{Remark}
\usepackage{url}

\author{Paul Myer Kominers and Scott Duke Kominers}
\title{Candy-passing Games on General Graphs, II}
\thanks{The second author gratefully acknowledges the support of a Harvard Mathematics Department Highbridge Fellowship.}
\address{\newline\indent Student, Department of Mathematics, Massachusetts Institute of Technology}
\email{pkoms@mit.edu}
\address{Student, Department of Mathematics, Harvard University\newline\indent c/o 8520
  Burning Tree Road\newline \indent Bethesda, MD 20817}
\email{kominers@fas.harvard.edu}


\newcommand\pass[2]{\varphi_{#1}(#2)}

\begin{document}
\begin{abstract}
We give a new proof that any candy-passing game on a graph $G$ with at least $4|E(G)|-|V(G)|$ candies stabilizes.
Unlike the prior literature on candy-passing games, we use
methods from the general theory of chip-firing games which allow us to
obtain a polynomial bound on the number of rounds before stabilization. 
\end{abstract}
\maketitle
\section{Introduction}
We let $G$ be an undirected graph and respectively denote the vertex and edge sets of $G$ by $V(G)$ and
$E(G)$.  The \emph{candy-passing game on $G$} is defined by the following rules:
\begin{itemize}
\item At the beginning of the game, $c>0$ candies are distributed
  among $|V(G)|$ students, each of whom is seated at some distinct vertex
  $v\in V(G)$.
\item A whistle is sounded at a regular interval.
\item Each time the whistle is sounded, every student who is able to do
  so passes one candy to each of his neighbors.  (If at the beginning of
  this step a student holds fewer candies than he has neighbors, he
  does nothing.)
\end{itemize}
Tanton \cite{Candy Passing} introduced this game for cyclic $G$. The
authors \cite{PS} extended the game to general graphs $G$.

The candy-passing game on $G$ is a special case of the well-known \emph{chip-firing
  game} on $G$ introduced by Bj\"orner, Lov\'asz, and Shor
  \cite{BLS}. Furthermore, terminating candy-passing games on $G$ are actually
  equivalent to terminating chip-firing games on $G$, by the following key theorem:
\begin{theorem}[\cite{BLS}]\label{t1}The initial configuration of a chip-firing
  game on $G$  determines whether the game will terminate.  If the game does
  terminate, then both the final configuration and length of the game
  are dependent only on the initial configuration.
\end{theorem}

Terminating chip-firing games have been studied extensively and are
surprisingly well-behaved.  In addition to Theorem \ref{t1}, it is known
that terminating chip-firing procesesses finish in polynomial time (see
\cite{T}).  Chip-firing games also have important applications; notably,
they are related to Tutte polynomials (see \cite{L}) and the critical
groups of graphs (see \cite{B}).

Infinite chip-firing games have received less attention, as the notion
of an ``end state'' of such a game is ambiguous.  By contrast, an infinite
candy-passing game admits a clear stabilization condition: the game
is said to have \emph{stabilized} if the configuration of candy will
never again change.

The first author \cite{PKoms} studied the end
  behavior of candy-passing games on $n$-cycles, proving the eventual
  stabilization of any candy-passing game on an $n$-cycle with at least
  $3n-2$ candies.  The authors \cite{PS} extended this analysis to arbitrary connected
  graphs $G$, showing that any candy-passing game on such $G$ with at
  least $4|E(G)|-|V(G)|$ candies will stabilize.

Here, we give a new proof of the stabilization result for general
connected graphs, using methods which allow us to obtain a polynomial
bound on the stabilization time.  Our approach draws from the literature
on chip-firing, using in particular a key result from Tardos's \cite{T} proof that
terminating chip-firing games conclude in polynomial time.

\section{The Setting}
As in the earlier work on candy-passing games, we refer to the interval
between soundings of the whistle as a \emph{round} of
candy-passing.   We
denote by $\pass{t}{v}$ the total number times a vertex $v\in V(G)$ has passed
candy by the end of round $t$.

Since infinite candy-passing games differ from infinite
chip-firing games, we will continue to distinguish between ``candies''
and ``chips.''  However,  we drop the student metaphor, treating the candy
piles as belonging to the vertices of the graph $G$.  For consistency, we
denote the total number of candies in a candy-passing game by $c$ throughout.  

Abusing terminology slightly, we say that a vertex has \emph{stabilized}
in some round if, after that round, the amount of candy held by that
vertex will not change during the remainder of the game.

For a vertex $v\in V(G)$, we denote the degree of $v$ by $\deg(v)$. We
say that a vertex $v\in V(G)$ is \emph{abundant} if it holds at least
$2\deg(v)$ pieces of candy.

Any vertex $v\in V(G)$ with $k\geq \deg(v)$ candies at the
beginning of a round passes $\deg(v)$ pieces of candy to its
neighbors and can, at most, receive one piece of candy from each of its
$\deg(v)$ neighbors.  Thus, such a vertex cannot end the round with more
than $k$ candies.  In particular, then, the set of abundant vertices
of $G$ can only shrink over the course of  a candy-passing game on $G$.

\section{Main Theorem}
We will prove the following stabilization theorem:
\begin{theorem}\label{thmP}
Let $G$ be a connected graph with diameter $d$. In any candy-passing game on $G$ with
$$c\geq4|E(G)|-|V(G)|$$ candies, every vertex $v\in V(G)$ will stabilize
within $|V(G)|\cdot d\cdot c$ rounds.
\end{theorem}
The stabilization component of Theorem \ref{thmP} was obtained in
\cite[Theorem 2]{PS}.  Our methods are inspired by those of Tardos
\cite{T}; they are essentially independent of the arguments used in
\cite{PKoms} and \cite{PS}.

We use the following lemma, which is a special case of Tardos's \cite{T} Lemma 5:
\begin{lemma}\label{l1}Let $v,v'\in V(G)$ be adjacent vertices of $G$.  Then,
  $|\pass{t}{v}-\pass{t}{v'}|\leq c$ for all $t$.\label{totals}
\end{lemma}
\noindent Additionally, we need an observation about the condition $c\geq 4|E(G)|-|V(G)|$.
\begin{lemma}\label{l2}For $G$ a graph and $c\geq 4|E(G)|-|V(G)|$, in any chip-firing game on $G$ with $c$ candies there is at least one vertex
  $v_*\in V(G)$ which passes candy every round.
\end{lemma}
\begin{proof}
It suffices to find a vertex $v_*\in V(G)$ which passes candy
  every round $t$ during which some vertex $v\in V(G)$ holds fewer
  than  $2\deg(v)-1$ candies.  

As observed above, it is not possible for a vertex $v\in V(G)$ which is not abundant at
  the beginning of round $t$ to become abundant after round $t$.
  However,  the condition $$c\geq 4|E(G)|-|V(G)|$$
  guarantees that whenever some $v\in V(G)$ holds fewer than $2\deg(v)-1$ candies there is also at
  least one abundant vertex $v'\in V(G)$.
 The existence of some vertex $v_*\in V(G)$ which is abundant in
  every round when some vertex $v\in V(G)$ has fewer than $2\deg(v)-1$
  candies then follows immediately.
\end{proof}
\begin{remark}Lemma \ref{l2} is, in some sense, dual to Tardos's \cite{T} Lemma 4
which shows that for any terminating chip-firing game on $G$ there is a distinguished
vertex $v_*\in V(G)$ which never fires.\end{remark}
We may now proceed with the proof of our main result:
\begin{proof}[Proof of Theorem \ref{thmP}]
By Lemma \ref{l2}, there is some vertex $v_*\in V(G)$ which passes candy
every round.  Denoting the rounds by $t=1,2,\ldots$, we then have
$\pass{t}{v_*}=t$ for all rounds $t$.  By Lemma \ref{l1}, we then know that
$$|\pass{t}{v_*}-\pass{t}{v}|\leq d\cdot c$$ for all $t$ and $v\in
V(G)$.  Since $\pass{t}{v_*}$ is strictly increasing in $t$, no $v\in
V(G)$ may fail to pass candy for more than $d\cdot c$ rounds.  In the
worst case, all but one vertex pass candy in each round when some vertex
does not pass candy; hence after $|V(G)|\cdot d\cdot c$ rounds all the
vertices of $G$ pass candy every round.
\end{proof}

\end{document}